\documentclass[12pt]{amsart}
\usepackage{amsmath, amssymb, amstext, amsthm, eqnarray, amscd}
\usepackage{tikz}
\usepackage{tikz-cd}
\usetikzlibrary{%
  matrix,%
  calc,%
  arrows%
}
\usepackage{fullpage}

\usepackage{float}
\usepackage{mathtools}
\usepackage{mathabx}
\usepackage{graphicx}
\usepackage[mathscr]{euscript}
\usepackage{enumerate}
\usepackage{hyperref}
\usepackage{pgfkeys}
\usepackage{fullpage}
\usepackage{pst-node}
\usepackage{lastpage}
\usepackage{fancyhdr}
\usepackage{multirow}
\allowdisplaybreaks
\usepackage{graphicx}
\usepackage{tikz,color,soul}
\usepackage{chngcntr}
\usepackage{stmaryrd}
\usepackage{tabularx}

\usepackage{enumitem}
\usepackage{kantlipsum}
\usepackage{array}
\usepackage{cleveref}
\usepackage[english]{babel}
\usepackage{enumitem}
\def\multiset#1#2{\ensuremath{\left(\kern-.3em\left(\genfrac{}{}{0pt}{}{#1}{#2}\right)\kern-.3em\right)}}

\newtheorem{theorem}{Theorem}[section]

\newtheorem{lem}[theorem]{Lemma}
\newtheorem{cor}[theorem]{Corollary}
\theoremstyle{definition}
\newtheorem{definition}[theorem]{Definition}
\newtheorem{eg}[theorem]{Example}
\newtheorem{remark}[theorem]{Remark}

\numberwithin{equation}{section}

\newcommand{\nothing}[1][]{%
    \tikz[overlay,remember picture]{
    \draw[red]
      ($(l4)+(-0.7em,0.7em)$) rectangle
      ($(r4)+(0.5em,-0.5em)$);}
}

\title{Invertibility of Anticommutator and Commutators of Higher Degree of $n$-potent Elements}

\author{Vivek Bhabani Lama}
   \address{Department of Mathematics, Indian Institute of Technology Kharagpur, Kharagpur, INDIA - 721302.}
   \email{ v.bhabani.lama@gmail.com, vivekbhabanilama@kgpian.iitkgp.ac.in}

\author{Suhas B N }
    \address{Department of Mathematics, The National Institute of Engineering (Autonomous under Visveswaraya Technological University), Mysuru, India}
    \email{suhasbn@nie.ac.in, suhasbn.sjc@gmail.com}


\date{}

\makeindex
\begin{document}
\maketitle


\section*{Abstract}
We introduce and study the notion of commutators and anti-commutator of higher degrees for ring elements, which generalize the concept of commutator and anti-commutator of ring elements. In particular, we study the invertibility of the degree $n$ commutators and anticommutator of $n$-potent elements. Under natural conditions on the ring, we relate the invertibility of degree $n$ commutators and anticommutator of $n$-potent elements. We also relate the invertibility of degree $n$ commutators and anticommutator of $n$-potent elements with the invertibililty of higher commutators and anticommutator. Finally, we study ring extensions in which the invertibility of degree $n$ commutators and anticommutator of $n$-potent elements is inherited from its base ring. 

\section{Introduction}

For an associative ring $S$ with unity, the commutators and anti-commutators encode important structural information about $S$. As is well known, the commutator of two elements $a,b \in S$, denoted by $[a,b]$, is the element $ab-ba$, and measures how far the two elements are from commuting. It is easy to see from the very definition that $[a,b] = -[b,a]$, conveying the anti-symmetric nature of the commutator. The anti-commutator of $a,b$ on the other hand, is the symmetric analogue of the commutator. More precisely, it refers to the element $ab+ba$, and is denoted by $\langle a,b \rangle$. A direct computation shows that $ab = \frac{1}{2}([a,b]+\langle a,b \rangle)$, whenever $2$ is invertible in the base ring. For preliminaries on commutators and anti-commutators, one can see for instance \cite{kaufman2014commutators}, \cite{Khurana01072012} and \cite{khurana2018idempotents}. 

In recent years, considerable attention has been devoted to understanding how invertibility conditions on commutators and anti-commutators influence the structure of rings, leading to several striking characterization theorems for matrix rings and related classes of rings. In \cite{khurana2018idempotents}, Khurana and Lam initiated the study of the invertible commutators and anti-commutators of idempotent elements in a ring, motivated by an excercise of Kaplansky concerning a pair of idempotents with an invertible commutator (\cite[p. 25]{kaplansky1968rings}). They introduced two ring-theoretic properties: the \textit{Property $K$}, requiring the existence of idempotents whose commutator is invertible, and the \textit{Property $\overline{K}$}, requiring the existence of idempotents whose anti-commutator is invertible. Using Bott--Duffin invertibility (\cite{Bott1953OnTA}) together with several identities for idempotents due to Kato (\cite[(I.4.34), (I.4.44)]{kato1995perturbation}) and Koliha--Rakočević (\cite{koliha2002invertibility} and \cite{koliha2003invertibility}), they established fundamental relationships between these two properties, showing in particular that Property $K$ always implies Property $\overline{K}$. 

One of the main results of their article  is that they provide complete characterization of rings satisfying Property $K$: such rings are precisely the matrix rings $M_{2}(S)$ over rings S for which the multiplicative identity is expressible as a sum of two units. Building on this characterization, they also determine exactly when matrix rings over local or commutative rings satisfy Properties $K$ and 
$\overline{K}$, establish several equivalent formulations of these properties, and derive new criteria for $2 \times 2$ matrix rings in terms of invertible commutators and anti-commutators involving idempotent, nilpotent and involutory elements. 

Motivated by their results, in the present article, we introduce the notion of \textit{degree n commutators and anti-commutator} (see Definition \ref{DEF: Anticommutator of Higher degree}) of ring elements and study their invertibility for $n$-potent elements (any element $x$ in a ring is called an $n$-potent if $x^n = x$). An interesting result in \cite{khurana2018idempotents} is that the invertibility of the commutator of idempotent elements of a ring also implies the invertibility of their anti-commutator (see \cite[Theorem~2.13 (3)]{khurana2018idempotents}). Our first main result is a generalization of this result in which we relate the invertibility of the $n$-commutators ($[a,b]_n^L$ and $[a,b]_n^R$) of two $n$-potent elements with the invertibility of their $n$-anti-commutator ($<a,b>_n$). 

\vskip 2mm
\noindent
	{\bf Theorem A \bf (Theorem \ref{THM: Invertability of anticommutator vs commutator}).}
 \textit{  Let $S$ be a ring and $a,b$ be $n$-potent elements in $S$. Then, for any $n \geq 2$,  $a-b$ and $<a,b>_n$ are units in $S$ if and only if $a+b$, $[a,b]_n^L$ and $[a,b]_n^R$ are units in $S$.}
\vspace{2mm}

Theorem $A$ also recovers the result for idempotent elements in \cite[Theorem 2.13 (3)]{khurana2018idempotents} as a corollary (see \Cref{COR: Corollary for anticommutator vs commutator for n=2}). Next, we prove that an $n$-potent element automatically retains its potency at infinitely many higher powers, revealing a recurring pattern in its behavior (see \Cref{LEM: when is x^k=x?}). Building upon this, we relate the invertibility of the $n$-commutators of two $n$-potent elements with the invertibility of their higher degree commutators.

\vskip 2mm
\noindent
	{\bf Theorem B \bf (Theorem \ref{THM: Higher vs Lower}).}
 \textit{     Let $S$ be a ring and $a,b$ be two $n$-potent elements in $S$ such that $<a,b>_n$ , $[a,b]_n^L$ and $[a,b]_n^R$ are units in $S$. Then the following results hold:
    \begin{enumerate}
        \item For every $m >n$ such that $a^m=a$ and $b^m=b$, the anti-commutator of degree $m$, $<a,b>_m$, is a unit in $S$ if and only if $\displaystyle\sum_{k=0}^{m-2}(a+b)^k$ is a unit in $S$.
        \item For every $m >n$ such that $a^m=a$ and $b^m=b$, the commutators of degree $m$,  $[a,b]_m^L$ and $[a,b]_m^R$, are units in $S$ if and only if $\displaystyle\sum_{k=0}^{m-2}(a+b)^k$ is a unit in $S$. 
    \end{enumerate}}
\vspace{2mm}

Finally, we study the invertibility of $n$-commutators and $n$-anticommutator of two $n$-potent elements in rings with only trivial idempotents and give particular insight to the invertibility of $3$-commutators and $3$-anticommutator of tripotent elements in their matrix ring of size $2\times 2$ (see \Cref{THM: Commutators and Anticommutators in Rings with only trivial idempotents} and \Cref{THM: Commutators and Anticommutators in matrix rings over Rings with only trivial idempotents}). Analogous to the properties $K$ and $\overline{K}$, we define properties $K_n$ and $\overline{K_n}$ and prove that these properties in certain subrings of matrix rings are inherited from their base rings (see \Cref{THM: Property Kn for D_m(S)}, \ref{THM: Property Kn for T_m(S)}  and \ref{THM: Property Kn for triangular matrix rings}). \\

\textbf{Organization of the paper.} In
Section~\ref{Section S2}, we introduce the notion of commutators and anti-commutator of higher degrees for ring elements. Further, we  prove our main results \Cref{THM: Invertability of anticommutator vs commutator} and \Cref{THM: Higher vs Lower}. Finally, in Section \ref{Section S3}, we define Properties $K_n$ and $\overline{K}_n$ and show that certain subrings of matrix rings inherit these properties from their base rings (\Cref{THM: Property Kn for D_m(S)}, \Cref{THM: Property Kn for T_m(S)}  and \Cref{THM: Property Kn for triangular matrix rings}). 

\section{Commutators and Anti-commutators of Higher Degrees}\label{Section S2}
 In this section, we introduce the notion of commutators and anti-commutator of higher degrees of ring elements. In particular, we focus on the invertibility of commutators and anti-commutator of higher degrees of $n$-potent elements. From now on, by a ring, we mean an associative ring with unity (not necessarily commutative) unless stated otherwise.
 


\begin{definition}\label{DEF: Anticommutator of Higher degree}
Let $S$ be a ring and $a,b \in S$. The \textit{anti-commutator of $a\text{ and }b$ of degree $n \ge 2$} is given by:
   $$<a,b>_n :=(a+b)^n-a^n-b^n$$
\end{definition}
\begin{remark} \label{remark on anti-commutator}
    The following are easy to observe:
    \begin{enumerate}
        \item Let $a,b $ be any two elements in $S$. Then, for any $n \geq 2$, the anti-commutator of $a\text{ and }b$ of degree $n$ can also be expressed as \[<a,b>_n = \sum\limits_{k=1}^{n-1}W_{n,k}(a,b),\]
where $W_{n,k}(a,b)$ refers to all words in $a$ and $b$ of length $n$ in which $a$ appears exactly $k$ times.
 
    \item Let $a,b $ be any two elements in $S$. For $n=2,$ the anti-commutator of $a \text{ and }b$ of degree $2$ is nothing but the element $ab+ba$, the anti-commutator of$\ a \text{ and }b$ in the usual sense. In our notation, $<a,b>_2=(a+b)^2-a^2-b^2=ab+ba=<a,b>$. 
    \item Suppose that $a$ and $b$ are two $n$-potents in $S$. Then from the first remark it follows that 
    \begin{eqnarray*}
        <a,b>_n &=& (a+b)((a+b)^{n-1}-1) \\
                &=& (a+b)(a+b-1)\sum\limits_{k=0}^{n-2}(a+b)^k.
    \end{eqnarray*}
    
 \end{enumerate}
\end{remark}

While Definition \ref{DEF: Anticommutator of Higher degree} generalizes nicely for anti-commutators of higher degrees, an analogous definition for commutators of higher degree is slightly more complicated. We now define two ``commutators of higher degree" that generalize the standard notion of commutators found in literature.
\begin{definition}\label{DEF: commutators of higher degree}  Let $a,b $ be any two elements in $S$.  Then, for any $n \geq 2$,
    \begin{enumerate}
       \item The left commutator of $a\text{ and }b$ of degree $n$ is denoted by $[a,b]_n^L$, and is defined by
    $$[a,b]_n^L =(a-b)(a+b)^{n-1}-(a^n-b^n).$$
       \item The right commutator of $a\text{ and }b$ of degree $n$ is denoted by $[a,b]_n^R$, and is defined by
  $$[a,b]_n^R =(a+b)^{n-1}(a-b)-(a^n-b^n).$$
         
    \end{enumerate}
\end{definition}

\begin{remark}\label{REM: Remark on commutators of higher degree}
    Analogous to Remark \ref{remark on anti-commutator}, the following are easy to observe:
    \begin{enumerate}
     \item  The left commutator of $a \text{ and } b$ of degree $n$ can also be expressed as $$[a,b]_n^L=\sum\limits_{k=1}^{n-1}U_{n,k}(a,b) - \sum\limits_{k=1}^{n-1}V_{n,k}(a,b),$$ where $U_{n,k}(a,b)$ refers to all words of length $n$ in $a$ and $b$ beginning with $a$, in which $a$ appears exactly $k$ times and $V_{n,k}(a,b)$ refers to all words of length $n$ in $a$ and $b$ beginning with $b$, in which $b$ appears exactly $k$ times.
          \item  The right commutator of $a \text{ and } b$ of degree $n$ can also be expressed as $$[a,b]_n^R=\sum\limits_{k=1}^{n-1}\widetilde{U_{n,k}}(a,b) - \sum\limits_{k=1}^{n-1}\widetilde{V_{n,k}}(a,b),$$ where $\widetilde{U_{n,k}}(a,b)$ refers to all words of length $n$ in $a$ and $b$ ending with $a$, in which $a$ appears exactly $k$ times and $\widetilde{V_{n,k}}(a,b)$ refers to all words of length $n$ in $a$ and $b$ ending with $b$, in which $b$ appears exactly $k$ times.
    \item Let $a,b $ be any two elements in $S$. For $n=2,$ the left commutator of $a \text{ and }b$ of degree $2$ is exactly the commutator of$\ a \text{ and }b$ and the right-commutator of $a \text{ and }b$ of degree $2$ is exactly the commutator of$\ b \text{ and }a$ in the usual sense. \\ In our notation, we have the following:
    \begin{enumerate}
        \item $[a,b]_2^L=(a-b)(a+b)-(a^2-b^2)=ab-ba=[a,b]. $
        \item $[a,b]_2^R= (a+b)(a-b)-(a^2-b^2)=ba-ab=[b,a].$
    \end{enumerate} Trivially, we notice that $[a,b]_2^L$ is a unit in $S$ if and only if  $[a,b]_2^R$ is a unit in $S$. 
 \end{enumerate}
\end{remark}
We provide a simple example to illustrate the definitions of anti-commutator and commutators of higher degree and the above Remarks \ref{remark on anti-commutator} and \ref{REM: Remark on commutators of higher degree}. 

\begin{eg}\label{EG: Example to show definition and remark are same for higher degree anticommutators, commutators}
    Let $n=3$ and $a,b$ be two elements in a ring $S$. Then
    \begin{eqnarray*}
        W_{3,1}(a,b) = ab^2+b^2a+bab &,& W_{3,2}(a,b) = a^2b+ba^2+aba.
        \end{eqnarray*}
        \begin{eqnarray*}
        \text{So}~\langle a,b \rangle_3 &=& W_{3,1}(a,b)+W_{3,2}(a,b) \\
        &=& (ab^2+b^2a+bab) + (a^2b+ba^2+aba).
    \end{eqnarray*}
    Similarly, 
    \begin{eqnarray*}
    U_{3,1}(a,b) = ab^2&,& 
    U_{3,2}(a,b)= a^2b+aba, \\
    V_{3,1}(a,b)= ba^2&,&  
    V_{3,2}(a,b) = b^2a+bab. \\
        \text{So}~ \left[a,b\right]_n^L &=& (U_{3,1}(a,b)+U_{3,2}(a,b)) - (V_{3,1}(a,b) + V_{3,2}(a,b)) \\
        &=& (ab^2+a^2b+aba)-(ba^2+b^2a+bab).
    \end{eqnarray*}
      And for the right commutator,
    \begin{eqnarray*}
        \widetilde{U_{3,1}}(a,b) = b^2a&,& 
        \widetilde{U_{3,2}}(a,b) = ba^2+aba, \\
        \widetilde{V_{3,1}}(a,b) = a^2b&,& 
        \widetilde{V_{3,2}}(a,b) = ab^2+bab.\\
         \text{So}~\left[a,b\right]_{n}^R &=&(\widetilde{U_{3,1}}(a,b)+\widetilde{U_{3,2}}(a,b)) - (\widetilde{V_{3,1}}(a,b) + \widetilde{V_{3,2}}(a,b)) \\
        &=&(b^2a+ba^2+aba) - (a^2b+ab^2+bab).    
        \end{eqnarray*}
    
\end{eg}

Next, we prove our first main theorem that relates the invertibility of the degree $n$ anticommutator of two $n$-potent elements $a,b \in S$ to the invertibility of the left and right commutators of degree $n$ of the elements $a$ and $b$. 
\begin{theorem}\label{THM: Invertability of anticommutator vs commutator}
 Let $S$ be a ring and $a,b$ be $n$-potent elements in $S$. Then, for any $n \geq 2$,  $a-b$ and $<a,b>_n$ are units in $S$ if and only if $a+b$, $[a,b]_n^L$ and $[a,b]_n^R$ are units in $S$.  
\end{theorem}
\begin{proof}
 First, given that $a \text{ and } b$ are $n$-potent elements in $S$, we have the following expressions for $<a,b>_n$, $[a,b]_n^L$ and $[a,b]_n^R$:
 \begin{align*}
     <a,b>_n&=(a+b)^n-(a+b)=(a+b)((a+b)^{n-1}-1)\\
     [a,b]_n^L&=(a-b)(a+b)^{n-1}-(a-b)=(a-b)((a+b)^{n-1}-1)\\
     [a,b]_n^R&=(a+b)^{n-1}(a-b)-(a-b)=((a+b)^{n-1}-1))(a-b).
 \end{align*}
\noindent Notice that $ <a,b>_n$ is a unit in $S$ if and only if  $(a+b)$ and $((a+b)^{n-1}-1)$ are units in $S$ as both these factors commute with each other. However, in the case of commutators, the factors $(a-b)$ and $(a+b)^{n-1}-1$ do not commute with each other. So here $[a,b]_n^L$ is a unit implies $((a+b)^{n-1}-1)$ has a left inverse and $(a-b)$ has a right inverse. Similarly, $[a,b]_n^R$ is a unit implies $((a+b)^{n-1}-1)$ has a right inverse and $(a-b)$ has a left inverse. Therefore, if both $[a,b]_n^L$ and $[a,b]_n^R$ are units in $S$ then both $((a+b)^{n-1}-1)$ and $(a-b)$ are units in $S$. So if $a+b$, $[a,b]_n^L$ and $[a,b]_n^R$ are units in $S$, then we can conclude that $(a-b)$ and $ <a,b>_n$ are units in $S$. 

Conversely, suppose that $(a-b)$ and $ <a,b>_n$ are units in $S$. 
This gives that $(a-b)$, $(a+b)^{n-1}-1$ and $a+b$ are units in $S$. Let $u$ and $v$ be the inverses of $(a-b)$ and $(a+b)^{n-1}-1$ respectively. Then, it is clear that $vu$ is the inverse of $[a,b]_n^L$ and $uv$ is the inverse of $[a,b]_n^R$.
\end{proof}
\noindent The above theorem immediately recovers the following result for idempotent elements in \cite{khurana2018idempotents}.

\begin{cor}\cite[Theorem~2.13 (3)]{khurana2018idempotents} \label{COR: Corollary for anticommutator vs commutator for n=2}
Let $S$ be a ring and $a,b$ be idempotents in $S$. If $[a,b]_2^L=[a,b]$ is a unit in $S$ then $<a,b>_2=<a,b>$ is a unit in $S$. 
\end{cor}
\begin{proof}
  Given that $[a,b]_2^L=[a,b]$ is a unit in $S$, the fact that  $[a,b]_2^L$ is a unit in $S$ if and only if $[a,b]_2^R$ is a unit in $S$ and the proof of \Cref{THM: Invertability of anticommutator vs commutator} gives that $a-b$ is a unit in $S$. However, this implies that $a+b$ is a unit in $S$ (see \cite[Theorem~2.8]{khurana2018idempotents}). Therefore, by \Cref{THM: Invertability of anticommutator vs commutator} we have $<a,b>_2=<a,b>$ is a unit in $S$. 
\end{proof}

For an $n$-potent element $x \in S$ we show that there are infinitely many choices of $m \in \mathbb{N}$ such that $x$ is also a $m$-potent element. We then relate the invertibility of $<a,b>_n$ with $<a,b>_m$ and the invertibility of $[a,b]_n^L$ and $[a,b]_n^R$ with the invertibility of $[a,b]_m^L$ and $[a,b]_m^R$ for $n$-potents $a$ and $b$ (see \Cref{THM: Higher vs Lower}).

\begin{lem}\label{LEM: when is x^k=x?}
    Let $S$ be a ring and $x$ be an $n$-potent element in $S$ such that the elements of the set  $ \{x,x^2,\dots,x^{n-1}\}$ are all distinct. Then, $x^{m}=x$ if and only if $m \equiv 1 \pmod{n-1}$.
\end{lem}
\begin{proof}
Firstly, observe that $x^n=x$ implies $x^{n-1}\cdot x=x^n=x$. Hence, for any $k \geq 1$, $x^{n-1}x^k=x^k$ and $x^{k({n-1})}=x^{n-1}$. Also, since $x,x^2,x^3\dots,x^{n-1}$ are all distinct elements, therefore the sequence $\{x^i\}_{i \geq 1}$ is periodic with period $n-1$. Now, suppose that $x^m=x$ then by division algorithm, we have: $$m=q(n-1)+r \text{ where } 1 \leq r<n-1$$
Therefore, $x^{m}=x^{q(n-1)+r}=x^{r} \text{ where } 1 \leq r<n-1.$ However, we have $x^m=x$ and $x,x^2,\dots,x^{n-1}$ are distinct elements. Therefore, $r=1$ and hence $m \equiv 1 \pmod{n-1}$. Conversely, suppose that $m \equiv 1 \pmod{n-1}$ then, $m -1=k(n-1)$ for some $k$. Therefore, $m=1+k(n-1)$ and hence $x^m=x^{1+k(n-1)}=x\cdot(x^{n-1})^k=x.$

\end{proof}
\begin{theorem}\label{THM: Higher vs Lower}
    Let $S$ be a ring and $a,b$ be two $n$-potent elements in $S$ such that $<a,b>_n$ , $[a,b]_n^L$ and $[a,b]_n^R$ are units in $S$. Then the following results hold:
    \begin{enumerate}
        \item For every $m >n$ such that $a^m=a$ and $b^m=b$, the anticommutator of degree $m$ given by $<a,b>_m$ is a unit in $S$ if and only if $\displaystyle\sum_{k=0}^{m-2}(a+b)^k$ is a unit in $S$.
        \item For every $m >n$ such that $a^m=a$ and $b^m=b$, the commutators of degree $m$ given by $[a,b]_m^L$ and $[a,b]_m^R$ are units in $S$ if and only if $\displaystyle\sum_{k=0}^{m-2}(a+b)^k$ is a unit in $S$. 
    \end{enumerate}
    
\end{theorem}
\begin{proof}
We prove the given statement for the case of anticommutators, as the proof for the case of commutators follows using similar arguments. Given $m>n \geq 2,$ we have the following expressions for $<a,b>_n$ and $<a,b>_m$,
$$<a,b>_n=(a+b)((a+b)^{n-1}-1) \text{ and }<a,b>_m=(a+b)((a+b)^{m-1}-1).$$
Now, $((a+b)^{n-1}-1)=(a+b-1)\left(\displaystyle\sum_{k=0}^{n-2}(a+b)^k\right)$, and hence if $<a,b>_n$ is a unit in $S$, then $(a+b),(a+b-1) \text{ and }\left(\displaystyle\sum_{k=0}^{n-2}(a+b)^k\right)$ are units in $S$. Using similar arguments, for $m>n$, the expression for $<a,b>_m$ is given by:
$$<a,b>_m=(a+b)(a+b+1)\left(\displaystyle\sum_{k=0}^{m-2}(a+b)^k\right).$$
It is clear that if $<a,b>_m$ is a unit in $S$ then $\left(\displaystyle\sum_{k=0}^{m-2}(a+b)^k\right)$ is also a unit in $S$ as it appears as a factor of $<a,b>_m$ that commutes with the other factors. Conversely, when $<a,b>_n$ and  $\left(\displaystyle\sum_{k=0}^{m-2}(a+b)^k\right)$ are units in $S$. Then, $(a+b), (a+b-1) \text{ and }\left(\displaystyle\sum_{k=0}^{m-2}(a+b)^k\right)$ are all units in $S$ and hence $<a,b>_m$ is a unit in $S$. 
\end{proof}
The following is an immediate corollary of \Cref{THM: Higher vs Lower} which focuses on the case of idempotents. 
\begin{cor}\label{COR: Higher vs Lower for Idempotents}
   Let $S$ be a ring and $a,b$ be two idempotent elements of $S$ such that $<a,b>_2$ and $[a,b]^L_2 \text{ or } [a,b]_2^R$ are units in $S$. Then the following results hold:
    \begin{enumerate}
        \item For every $m >2$, $<a,b>_m$ is a unit in $S$ if and only if $\displaystyle\sum_{k=0}^{m-2}(a+b)^k$ is a unit in $S$.
        \item For every $m >2$ , the commutators of degree $m$ given by $[a,b]_m^L$ and $[a,b]_m^R$ are units in $S$ if and only if $\displaystyle\sum_{k=0}^{m-2}(a+b)^k$ is a unit in $S$. 
    \end{enumerate}   
\end{cor}
\begin{proof}
Observe that for idempotent elements $a \text{ and }b \in S$, we have $a^m=a$ and $b^m=b$ for all $m>2$ (see Lemma \ref{LEM: when is x^k=x?}). Also due to Remark \ref{remark on anti-commutator}, we have that the invertibility of $[a,b]^L_2 \text{ or } [a,b]_2^R$ implies the invertibility of the other. The rest is a direct consequence of \Cref{THM: Higher vs Lower} by setting $n=2$. 
\end{proof}

Next, we focus on the case of rings with only trivial idempotents. We show that for a ring $S$ with only trivial idempotents, any $n$-potent in $S$ is either $0$ or a unit (see \Cref{LEM: n-potents in rings with only trivial idempotents}). Consequently, any tripotent in $S$ cannot be genuine (in the sense of \cite{cualuguareanu2018tripotents}). We provide necessary definitions and context wherever necessary. 
\begin{definition}\label{DEF: Ring with only trivial Idempotents}
    A ring $S$ is said to have only trivial idempotents if the only idempotents in $S$ are $0$ and $1$. 
\end{definition}
\begin{lem}\label{LEM: n-potents in rings with only trivial idempotents}
    Let $S$ be a ring with only trivial idempotents. Then, any $n$-potent in $S$ is either $0$ or a unit. 
\end{lem}
\begin{proof}
Let $x$ be an $n$-potent element in $S$. If $n=2
$ then $x$ is an idempotent and hence $x=0 \text{ or } x=1$. Suppose $n \geq 3$ then consider the following 
$$(x^{n-1})^2=x^{2n-2}=x^n\cdot x^{n-2}=x^{n-1}.$$
Therefore, $x^{n-1}$ is an idempotent in $S$. Since $S$ has only trivial idempotents, $x^{n-1}=0$ or $x^{n-1}=1$. If $x^{n-1}=0$ then $x^n=x=0$, and if $x^{n-1}=1$ then $x$ is a unit in $S$.
\end{proof}
\noindent Genuine tripotents were first introduced in \cite[Section~1]{cualuguareanu2018tripotents}, we include the definition for the sake of completeness. 
\begin{definition}\label{DEF: Genuine Tripotents}
    An element $r$ of a ring $S$ is called a genuine tripotent if the  following conditions hold:
    \begin{enumerate}
        \item $r^3=r$
        \item $r$ is not an idempotent or negative of an idempotent
        \item $r^2 \neq 1$ i.e. $r$ is not an order $2$ unit. 
        \end{enumerate}
       A ring is said to have no genuine tripotents if the set of genuine tripotents is an empty set.\end{definition} 

\begin{cor}\cite[Remarks]{cualuguareanu2018tripotents}\label{Cor: Ring with only trivial idempotents has no genuine tripotents}
    Let $S$ be a ring with only trivial idempotents. Then, $S$ has no genuine tripotents. 
\end{cor}
\begin{proof}
 Let $x \in S$ be a tripotent element. Since $S$ has only trivial idempotents, by Lemma \ref{LEM: n-potents in rings with only trivial idempotents}, $x$ is either $0$ or a unit. If $x=0$ then $x$ is not a genuine tripotent. Also, $x$ is a tripotent element i.e. $x^3=x$. Further if $x$ is a unit, we get $x^2=1$. In any case $x$ is not a genuine tipotent. Therefore, $S$ has no genuine tripotents. 
\end{proof}

Using Lemma \ref{LEM: n-potents in rings with only trivial idempotents}, we now study the invertibility of 
$<a,b>_n$, $[a,b]_n^L$ and $[a,b]_n^R$ over rings with only trivial idempotents. We provide equivalent conditions for the invertibility of 
$<a,b>_n$, $[a,b]_n^L$ and $[a,b]_n^R$ in terms of the invertibility of simpler ring elements (see \Cref{THM: Commutators and Anticommutators in Rings with only trivial idempotents}). In order to prove this result, we require the following lemma. 
\begin{lem}\label{LEM: Commutators and Anticommutators that include Unity}
    Let $S$ be a ring. 
\begin{enumerate}
    \item If $x \neq 1$ is a tripotent element in $S$, then $<1,x>_3$ is never a unit in $S$. 
    \item If $n\ge 4$ and $x$ is an $n$-potent element in $S$ such that $x^{n-1} \neq 1$, then $<1,x>_n$ is never a unit in $S$.
    \item In general, for $n \ge 4$, if $x$ is an $n$-potent element of $S$, then $<1,x>_n$ is a unit if and only if $x$, $(1+x)$ and $t$ are units in $S$, where $t = \sum\limits_{k=0}^{n-2}(1+x)^{k}$.
\end{enumerate}
\end{lem}
\begin{proof}
    \begin{enumerate}
        \item For $n=3$, we have the following expression for $<1,x>_3$,
        \begin{align*}
           <1,x>_3  = (1+x)((1+x)^2-1) = (1+x)(x)(2+x).
        \end{align*}
         We also have $x^3 = x$ which implies that
            $x(x^2-1) = 0$, and so $x(x-1)(x+1) = 0.$ Therefore, $<1,x>_3$ is a unit implies $(1+x)$, $x$ and $(2+x)$ are units, and this forces $(x-1)$ to be zero. This is a contradiction to the hypothesis. 
             \item For $n \geq 4$,  we have the following expression for $<1,x>_n$,
             \begin{align*}
              <1,x>_n & = (1+x)((1+x)^{n-1}-1)   \\
              & = (1+x)(x)(t),~~~\text{where $t = \sum\limits_{k=0}^{n-2}(1+x)^k$.}
             \end{align*}
             So if $<1,x>_n$ is a unit, $x$ has to be a unit. As $x$ is an $n$-potent, we also have $x(x^{n-1}-1)=0$. This implies $x^{n-1} = 1,$ contradicting the hypothesis again. So $<1,x>_n$ can never be a unit if $x^{n-1} \neq 1.$
             \item From the previous subdivision, it is clear that $<1,x>_n$ is a unit if and only if $x$, $(1+x)$ and $t$ are units.
    \end{enumerate}
\end{proof}

\begin{theorem}\label{THM: Commutators and Anticommutators in Rings with only trivial idempotents}
   Let $S$ be a ring with only trivial idempotents and let $a$ and $b$ be $n$-potent elements in $S$. Then the following results hold:
   \begin{enumerate}
       \item If either $a=0$ or $b=0$, then $<a,b>_n=0$ and is never a unit in $S$.
       \item If $a=b=1$ then $<a,b>_n$ is a unit if and only if $2^n-2$ is a unit in $S$. 
       \item If $a=b=-1$ then $<a,b>_n$ is a unit if and only if $2-2^n$ is a unit if $n$ is odd and $2^n+2$ is a unit if $n$ is even. 
       \item If $a=-1 \text{ and }b=1$ or  $a=1 \text{ and }b=-1$ then $<a,b>_n=0$ is never a unit in $S$.
       \item If $a$ and $b$ are units such that $a\neq 1,-1$ and $b \neq 1,-1$ then $<a,b>_n$ is a unit if and only if $(a+b)$ and $1+<a,b>_{n-1}$ are units in $S$. 
   \end{enumerate}
\end{theorem}
\begin{proof}
    (1) and (4) are trivial due to the fact that each term in $<a,b>_n$ is zero for (1) and $a+b=0$ for (4). For any $n$-potents $a $ and $b$ we have:
    $$<a,b>_n=(a+b)((a+b)^{n-1}-1).$$
    In particular if $a=b=1$ then $<a,b>_n=2^n-2$. Hence, the result (2) is obtained. Further if $a=b=-1$ and $n$ is odd then $<a,b>_n=2-2^n$ and if $n$ is even then $<a,b>_n=2+2^n$. Therefore the result (3) is obtained. Finally for (5), we observe that if $a$ and $b$ are both units then $a^{n-1}=b^{n-1}=1$. Using this fact, we get the following expressions:
    \begin{align*}
        <a,b>_n=&(a+b)((a+b)^{n-1}-1)\\
        =& (a+b)(a^{n-1}+b^{n-1}+<a,b>_{n-1}-1)\\
        =&(a+b)(<a,b>_{n-1}+1)        
    \end{align*}
Therefore, $<a,b>_n$ is a unit if and only if $(a+b)$ and $1+<a,b>_{n-1}$ are units in $S$.
\end{proof}
Next, we extend the results of \Cref{THM: Commutators and Anticommutators in Rings with only trivial idempotents} to matrix rings of size $2$ over rings with only trivial idempotents but we discuss the case of tripotent elements.  

\begin{theorem}\label{THM: Commutators and Anticommutators in matrix rings over Rings with only trivial idempotents}
    Let $S$ be a commutative ring with only trivial idempotents. Then, for any $n$-potent element $A\in M_2(S)$ for any $m \geq 1$ we have that $\mathrm{det}(A)$ is either $0$ or a unit. Furthermore, the following statements hold regarding the invertibility of $<A,B>_3$, $[A,B]_3^R$ and $[A,B]_3^L$ where $A$ and $B$ are tripotents in $M_2(S)$:
    \begin{enumerate}
        \item If $\mathrm{det(A)}=\mathrm{det}(B)=0$ then $<A,B>_3$ is a unit in $M_2(S)$ if and only if $[A,B]_3^R$ and $[A,B]_3^L$ are units in $M_2(S).$\item If $\mathrm{det}(A) $  and $\mathrm{det}(A)$ are units in $S$ then one of the following hold:
        \begin{enumerate}
              \item If $[A,B]_2^L$ is a unit in $M_2(S)$  then $<A,B>_3$ is a unit in $M_2(S)$ if and only if $[A,B]_3^R$ and $[A,B]_3^L$ are units in $M_2(S).$
        \item If $[A,B]_2^L$  is not a unit in $M_2(S)$  then either $<A,B>_3$ or  $[A,B]_3^R$ and $[A,B]_3^L$ are not units in $M_2(S).$
        \end{enumerate}
        \item If $\mathrm{det}(A)=0$ , $\mathrm{det}(B)$ is a unit in $S$ then one of the following hold: \begin{enumerate}
            \item If $A^2-I-[A,B]_2^L$, $A^2-I-[A,B]_2^R$ are units in $M_2(S)$  then $<A,B>_3$ is a unit in $M_2(S)$ if and only if $[A,B]_3^R$ and $[A,B]_3^L$ are units in $M_2(S).$
        \item If any of $A^2-I-[A,B]_2^L$ or $A^2-I-[A,B]_2^R$ is not a unit in $M_2(S)$  then either $<A,B>_3$ or  $[A,B]_3^R$ and $[A,B]_3^L$ are not units in $M_2(S).$
        
        \end{enumerate}
  \item If $\mathrm{det}(B)=0$ , $\mathrm{det}(A)$ is a unit in $S$ then one of the following hold: \begin{enumerate}
            \item If $-B^2+I-[A,B]_2^L$, $-B^2+I-[A,B]_2^R$ are units in $M_2(S)$  then $<A,B>_3$ is a unit in $M_2(S)$ if and only if $[A,B]_3^R$ and $[A,B]_3^L$ are units in $M_2(S).$
        \item If any of $-B^3+I-[A,B]_2^L$ or $-B^3+I-[A,B]_2^R$ is not a unit in $M_2(S)$  then either $<A,B>_3$ or  $[A,B]_3^R$ and $[A,B]_3^L$ are not units in $M_2(S).$
        
        \end{enumerate}
        
    \end{enumerate}
\end{theorem}
\begin{proof}
   Firstly, suppose that $A$ is an $n$-potent element in $M_2(S)$. Then, $A^n=A$ and hence $\det(A^n)=\det(A)$. This reduces to $(\det(A))^n=\det(A)$ i.e. $\det(A)$ is an $n$-potent in $S$. Since $S$ has only trivial idempotents, $\det(A)$ is either $0$ or a unit in $S$ (see Lemma \ref{LEM: n-potents in rings with only trivial idempotents}). \newline For (1): under the condition that $\det(A)=\det(B)=0$ where $A$ and $B$ are matrices in $M_2(S)$, we have that $A+B$ is a unit if and only if $A-B$ is a unit (see \cite[Corollary~4.9]{khurana2018idempotents}). From the proof of \Cref{THM: Invertability of anticommutator vs commutator}, we get that $<A,B>_3$ is a unit if and only if $[A,B]_3^R$ and $[A,B]_3^L$ are units. \newline For (2): since $\det(A)$ and $\det(B)$ are both units in $S$. Therefore, $A$ and $B$ are units in $M_2(S)$ (see \cite[Theorem]{grinshpon2008invertibility} and \cite{khurana2009short} ). Also, since $A$ and $B$ are tripotents, this gives $A^2=B^2=I$. Consider the following equalities: 
   \begin{align*}
     (A+B)(A-B)=A^2-AB+BA-B^2=-[A,B]^L_2 \\
     (A-B)(A+B)=A^2+AB-BA-B^2=[A,B]^L_2
   \end{align*}From the above equalities, $[A,B]_2^L$ is a unit if and only if $(A+B)$ and $(A-B)$ are units. This in addition to \Cref{THM: Invertability of anticommutator vs commutator} gives the desired result.  
   \newline For (3) and (4): We provide a detailed proof of (3), as (4) is a symmetric case of (3) the proof follows using similar arguments. We just need to observe that the following equalities hold:
   \begin{align*}
      (A+B)(A-B)=&A^2-I-[A,B]_2^L\\
      (A-B)(A+B)=& A^2-I-[A,B]_2^R
   \end{align*}
 The desired result follows using the same line of argument as in the proof of part (2).
   \end{proof}
\section{The Properties\texorpdfstring{ $K_n$ and $\overline{K_n}$}{Kn and Kn-bar} on Subrings of Matrix Rings}\label{Section S3}
 In \cite{khurana2018idempotents}, the authors define Property $K$ and Property $\overline{K}$ as the existence of two idempotent elements in the ring such that their commutator and anticommutator are units respectively. This definition can be naturally extended to the case of $n$-potent elements as Property $K_n$ and Property $\overline{K_n}$ respectively (see Definitions \ref{DEF: Property Kn} and \ref{DEF: Property Kn bar}). In this section we show that these properties are inherited by the formal triangular rings, upper triangular $T_m(S)$ and upper triangular rings with equal diagonal entries $D_m(S)$ from their base rings.

\begin{definition}\label{DEF: Property Kn}
  For any $n \geq 2$, a ring $S$ is said to satisfy \textit{Property $K_n$} if there exists two $n$-potent elements $a,b \in S$ such that $[a,b]_n^L$ and $[a,b]_n^R$ are units in $S$.  
\end{definition}
\begin{definition}\label{DEF: Property Kn bar}
  For any $n \geq 2$, a ring $S$ is said to satisfy \textit{Property $\overline{K_n}$} if there exists two $n$-potent elements $a,b \in S$ such that $<a,b>_n$ is a  unit in $S$.  
\end{definition}

\begin{remark}\label{REM: K_n and K_n bar}
We observe the following:
\begin{enumerate}
    \item For a ring $S$, the properties $K_2$ and $\overline{K_2}$ correspond to properties $K$ and $\overline{K}$ respectively (in the sense of \cite{khurana2018idempotents}).
    \item Due to Corollary \ref{COR: Corollary for anticommutator vs commutator for n=2}, a ring satisfying Property $K_2$ aslo satisfies Property $\overline{K_2}$. This statement is independently proved in \cite[Theorem A]{khurana2018idempotents}. 
\end{enumerate}
\end{remark}


Let $S$ be a ring, $T_m(S)$ be the ring of $m \times m$ upper triangular matrices over $S$ and $D_m(S)$ denote the ring of $m \times m$ upper triangular matrices over $S$ with equal diagonal entries. 
$$D_m(S)=\big\lbrace (a_{ij}) \in T_m(S)\  |\ a_{11} =a_{22} = \cdots = a_{mm} \big\rbrace $$
We now show that the properties $K_n$ and $\overline{K_n}$ hold for the rings $D_m(S)$ and $T_m(S)$ if and only if they hold for the ring $S$ itself. 
\begin{theorem}\label{THM: Property Kn for D_m(S)}
    Let $S$ be a ring. Then the following statements hold:
    \begin{enumerate}
        \item For any $m \geq 1$ and any $n \geq 2$, $D_m(S)$ satisfies Property $K_n$ if and only if $S$ satisfies Property $K_n$. 
        \item For any $m \geq 1$ and any $n \geq 2$, $D_m(S)$ satisfies Property $\overline{K_n}$ if and only if $S$ satisfies Property $\overline{K}_n$.  
    \end{enumerate}
   
\end{theorem}
\begin{proof}
Firstly, we notice that for any $n$-potent element $A=(a_{ij})_{m \times m} \in D_m(S)$ we have $A^n=A$ and hence the diagonal entries $a_{ii}=a$ for all $i=1,2,\dots,m$ is an $n$-potent element in $S$. We only provide the proof of (1) as (2) can be obtaied by similar arguments.\\
For (1): If $A=(a_{ij})_{m \times m}$ and $B=(b_{ij})_{m \times m}$ are $n$-potent elements in $D_m(S)$ with $a_{ii}=a$ and $b_{ii}=b$ for all $i=1,2,\dots,m$ such that $<A,B>_n$ is a unit in $D_n(S)$. Then, since both $A$
 and $B$ are upper triangular matrices, we have that $<a,b>_n$ is a unit in $S$. However, due to our observation, both $a$ and $b$ are $n$-potents in $S$ and hence, $S$ satisfies Property $K_n$. Converserly, if $S$ satisfies Property $K_n$ then there exists $n$-potent elements $a,b \in S$ such that $<a,b>_n$ is a unit in $S$. Consider the matrices $A=a\cdot I_m$ and $B=b\cdot I_m$ where $I_m$ is the identity matrix of order $m$. By the definition of $D_m(S)$, both $A,B \in D_m(S)$. We also have the following expression for the degree $n$ anticommutator of $A \text{ and }B$:
 $$<A,B>_n=<a\cdot I_m, b\cdot I_m>_n=<a,b>_n\cdot I_m.$$
Since $<a,b>_n$ is a unit in $S$, therefore $<A,B>_n$ is a unit in $D_m(S)$.
 \end{proof}
The following lemma is a simple observation regarding the lifting of Property $K_n$ and Property $\overline{K_n}$ from subrings to the original ring that will help us to simplify the proof of \Cref{THM: Property Kn for T_m(S)}. 

\begin{lem}\label{LEM: Kn and Kn bar on subrings}
    Let $S_1$ be a ring and $S_2$ be a subring of $S_1$ such that the unity of both the rings are same i.e. $1_{S_1}=1_{S_2}$. Then for any $n \geq 2$, if $S_2$ satisfies Property $K_n$ (Property $\overline{K_n})$ then $S_1$ satisfies Property $K_n$ (Property $\overline{K_n})$.
\end{lem}
\begin{proof}
Let $n > 1$, we prove the theorem for Property $K_n$ as the statement for Property $\overline{K_n}$ follows using exactly similar arguments.
    Suppose that $S_2$ satisfies property $K_n$ then there exists $n$-potent elements $a,b \in S_2$ such that $<a,b>_n$ is a unit in $S$. Therefore, there exists an element $s \in S_2$ such that the following equation holds:
    $$s \cdot <a,b>_n=<a,b>_n\cdot s=1_{S_2}$$
    However, since $S_2$ is a subring of $S_1$, the element $s \in S_1$ and $1_{S_2}=1_{S_1}$ we have: 
    $$s \cdot <a,b>_n=<a,b>_n\cdot s=1_{S_1}$$
    Therefore, $S_1$ satisfies Property $K_n$. 
\end{proof}
\begin{theorem}\label{THM: Property Kn for T_m(S)}
    Let $S$ be a ring. Then the following statements hold:
    \begin{enumerate}
        \item For any $m \geq 1$ and any $n \geq 2$, $T_m(S)$ satisfies Property $K_n$ if and only if $S$ satisfies Property $K_n$. 
        \item For any $m \geq 1$ and any $n \geq 2$, $T_m(S)$ satisfies Property $\overline{K_n}$ if and only if $S$ satisfies Property $\overline{K}_n$.  
    \end{enumerate}
   
\end{theorem}
\begin{proof}
Similar to \Cref{THM: Property Kn for D_m(S)}, we only give the proof of (1) as the proof of (2) follows using the same arguments. Again, for any $n$-potent element $A=(a_{ij})_{m \times m} \in T_m(S)$ we have $A^n=A$ and hence the diagonal entries $a_{ii}$ for all $i=1,2,\dots,m$ is an $n$-potent element.  If $A=(a_{ij})_{m \times m}$ and $B=(b_{ij})_{m \times m}$ are $n$-potent elements in $T_m(S)$  such that $<A,B>_n$ is a unit in $D_n(S)$. Then, since both $A$
 and $B$ are upper triangular matrices, we have that $<a_{ii},b_{ii}>_n$ is a unit in $S$ for all $i=1,2,\dots,m$. However, from our observation, each $a_{ii}$ and $b_{ii}$ are $n$-potent elements in $S$. Therefore, $S$ satisfies Property $K_n$. Conversely, suppose that $S$ satisfies Property $K_n$ then, the ring $D_m(S)$ also satisfies Property $K_n$ for all $m \geq 1$. Also, $D_m(S)$ is a subring of $T_m(S
)$ for all $m \geq 1$ with the property that $I_m=1_{D_m(S)}=1_{T_m(S)}$. Therefore by Lemma \ref{LEM: Kn and Kn bar on subrings}, $T_m(S)$ satisfies Property $K_n$ for all $m \geq 1$. 
\end{proof}
The following is an immediate corollary of the above theorem. 
\begin{cor}\cite[Proposition~2.18]{khurana2018idempotents}\label{COR: Property K2 for T_m(S)}
 Let $S$ be a ring then for any $m \geq 1$, the following hold:
 \begin{enumerate}
     \item $T_m(S)$ satisfies Property $K_2$ if and only if $S$ satisfies Property $K_2$. 
     \item $T_m(S)$ satisfies Property $\overline{K_2}$ if and only if $S$ satisfies Property $\overline{K_2}$.      
 \end{enumerate}
\end{cor}
\begin{proof}
    The corollary is a direct consequence of Theorem \ref{THM: Property Kn for T_m(S)} by setting $n=2$. 
\end{proof}

\begin{theorem}\label{THM: Property Kn for triangular matrix rings}
    Let $S_1$ and $S_2$ be  rings and $M$ be any $(S_1,S_2)$-bimodule and let $T=\begin{pmatrix}
            S_1&M\\
            0&S_2
        \end{pmatrix}$  be the formal triangular ring. Then the following statements hold:
    \begin{enumerate}
        \item For any $n \geq 2$, the ring $T$ satisfies Property $K_n$ if and only if $S_1$ and $S_2$ satisfy Property $K_n$. 
        \item For any $n \geq 2$, the ring $T$ satisfies Property $\overline{K_n}$ if and only if $S_1$ and $S_2$ satisfy Property $\overline{K_n}$. 
    \end{enumerate}
    \end{theorem}
 \begin{proof}
     Let $A=\begin{pmatrix}
         r&m\\
         0&s
     \end{pmatrix} \in T$ be an $n$-potent element. Due to the multiplication operation in $T$, the elements $r$ and $s$ are $n$-potent elements in $S_1$ and $S_2$ respectively. In order to prove the theorem, we give the proof of (1) as (2) follows using similar line of argument. Suppose that $A=\begin{pmatrix}
         r_1 & m_1\\
         0 & s_1
     \end{pmatrix}, B=\begin{pmatrix}
         r_2&m_2\\
         0&s_2
     \end{pmatrix} \in T$ are two $n$-potent elements such that $<A,B>_n$ is a unit in $T$. Using the multiplication in $T$, we have that $<r_1,r_2>_n$ is a unit in $S_1$ and $<s_1,s_2>_n$ is a unit in $S_2$. By our observation, the elements $r_1,r_2$ are $n$-potent in $S_1$ and $s_1,s_2$ are $n$-potent in $S_2$.  Conversely, let $r_1,r_2$ be $n$-potent in $S_1$ and $s_1,s_2$ be $n$-potent in $S_2$ such that $<r_1,r_2>_n$ is a unit in $S_1$ and $<s_1,s_2>$ is a unit in $S_2$. \text{Consider the elements} $ A=\begin{pmatrix}
         r_1 & 0\\
         0 & s_1
     \end{pmatrix} \text{ and } B=\begin{pmatrix}
         r_2&0\\
         0&s_2
     \end{pmatrix} \in T$, clearly both $A$ and $B$ are $n$-potent elements in $T$. Also, we have the following expression for $<A,B>_n$,
     $$<A,B>_n=\begin{pmatrix}
         <r_1,r_2>_n &0\\
         0&<s_1,s_2>_n
     \end{pmatrix} \in T$$
Since $<A,B>_n$ is a unit in $T$, $<r_1,r_2>_n$ is a unit in $S_1$ and $<s_1,s_2>_n$ is a unit in $S_2$. Therefore, $T$ satisfies Property $K_n$ if and only if $S_1$ and $S_2$ individually satisfy Property $K_n$. 
 \end{proof}
 The previous theorem recovers the following result from \cite{khurana2018idempotents} as an immediate corollary.
\begin{cor}\cite[Proposition~2.17]{khurana2018idempotents}\label{COR: Property K2 for triangular matrix rings}
    Let $R$ and $S$ be  rings and $M$ be any $(R,S)$-bimodule and let $T=\begin{pmatrix}
            R&M\\
            0&S
        \end{pmatrix}$  be the formal triangular ring. Then the following statements hold:
    \begin{enumerate}
        \item The ring $T$ satisfies Property $K_2$ if and only if $R$ and $S$ satisfy Property $K_2$. 
        \item The ring $T$ satisfies Property $\overline{K_2}$ if and only if $R$ and $S$ satisfy Property $\overline{K_2}$. 
    \end{enumerate}
\end{cor}
 \begin{proof}
     The proof is clear by setting $n=2$ in Theorem \ref{THM: Property Kn for triangular matrix rings}. 
 \end{proof}
 \section*{Acknowledgments}
 The first author is thankful to the Government of India for supporting him in this work through the Prime Minister Research Fellowship. 
 \bibliographystyle{abbrv}
\bibliography{refs}

\end{document}